\documentclass[a4paper,12pt
]{article}
\usepackage{amsmath,amssymb,amsthm}
\usepackage{mathtools}
\numberwithin{equation}{section}
\usepackage[margin=20mm]{geometry}
\usepackage{appendix}

\newtheorem{thm}{Theorem}[section]
\newtheorem{prop}[thm]{Proposition}
\newtheorem{cor}[thm]{Corollary}
\newtheorem{exam}[thm]{Example}
\newtheorem{rem}[thm]{Remark}
\newtheorem{lem}[thm]{Lemma}
\newtheorem{defn}[thm]{Definition}
\newtheorem{assum}[thm]{Assumption}

\begin{document}

\title{Conservativeness of time changed processes and \\
Liouville property for Schr\"odinger operators\thanks{This work was 
supported by JSPS KAKENHI Grant Numbers JP23K25773, JP24K06791.}}
\author{Yuichi Shiozawa\thanks{Department of Mathematical Sciences, 
Faculty of Science and Engineering, Doshisha University, 
Kyotanabe, Kyoto, 610-0394, Japan; \texttt{yshiozaw@mail.doshisha.ac.jp}} \ and 
Masayoshi Takeda\thanks{Department of Mathematics, Faculty of Engineering Science, 
Kansai University, Suita, Osaka, 564-8680, Japan; \texttt{mtakeda@kansai-u.ac.jp}}}

\maketitle

\begin{abstract}
We establish a criterion for the Liouville property for Schr\"odinger operators 
via the conservativeness of time changed processes. 
Using this criterion, we obtain necessary and sufficient conditions for the Liouville property 
for some Schr\"odinger operators 
in terms of the decay rates of the potentials at infinity/boundary. 
\end{abstract}

\section{Introduction}
We are concerned with the Liouville property for a Schr\"odinger operator of the form $-{\cal L}+\mu$, 
where ${\cal L}$ is a Markov generator without killing term (zeroth-order term) and $\mu$ is a non-negative measure.  
Here the Liouville property means that 
any bounded solution $h$ to the equation $(-{\cal L}+\mu)h=0$ must be zero.
Our purpose in this paper is to present a tractable condition for the Liouville property 
via the conservativeness of the $\mu$-time changed process of a symmetric Markov process 
generated by ${\cal L}$.

Grigor'yan-Hansen (\cite[Theorems 1.2 and 5.2]{GH98}) obtained conditions  
for the Liouville property for Schr\"odinger operators of the form $-{\cal L}+\mu$ on a Riemannian manifold, 
where ${\cal L}$ is the Laplace-Beltrami operator.
They further obtained explicit necessary and sufficient conditions 
for the Liouville property for concrete Schr\"odinger operators on ${\mathbb R}^d$
(\cite[Sections 6.2 and 6.3]{GH98}). 
Pinsky (\cite[Theorem 1]{P08}) also established, in a probabilistic way, a necessary and sufficient condition 
for the Liouville property for Schr\"odinger operators on ${\mathbb R}^d$ with $d\ge 3$, 
where ${\cal L}$ is the Laplace operator and the potential term is written as  
$\mu=V^+{\rm d}x-V^-{\rm d}x$ for some non-negative H\"older continuous functions $V^+$ and $V^-$. 
This condition is equivalent to the conservetiveness of the time changed Brownian motion with respect to $V^+$.
We here note that the arguments in \cite{GH98, P08} use  
the locality  of the Laplace(-Beltrami) operator 
or the sample path continuity of the Brownian motion. 

The second-named author (\cite{T18,T20,T25}) formulated 
the Liouville property for Dirichlet operators generated by general regular Dirichlet forms. 
In particular, it is proved in \cite[Corollary 2.9, Theorem 2.10 and Corollary 2.21]{T25} that 
the Liouville property is equivalent to the conservativeness of the $k$-time changed process  
of the resurrected process associated with the resurrected Dirichlet form, the form obtained by removing the killing part. 
Here $k$ is the killing measure in the Beurling-Deny formula.   
This result is regarded as an extension of \cite{GH98,P08} to more general Schr\"odinger operators
having non-local operators as ${\cal L}$.

In order to verify the Liouville property, we focus on the conservativeness of $\mu$-time changed processes. 
If the measure $\mu$ has full support, 
then the generator of the $\mu$-time changed process is formally written as ${\cal L}/\mu$, 
and the energy form for ${\cal L}$ remains unchanged by the $\mu$-time change. 
Using these properties, we obtained necessary and sufficient conditions for the conservativeness of 
time changed processes of symmetric stable-like processes and censored stable processes 
(\cite{S15, SU14}). 
On the other hand, the lifetime of the $\mu$-time changed process is   
the terminal time of the positive continuous additive functional (PCAF) corresponding to $\mu$. 
In particular, the divergence property of the latter is well studied 
for a class of stochastic processes such as L\'evy processes 
(see, e.g., \cite{KS20, KMS21, SV17} and references therein).

We will establish a sufficient condition 
for the almost sure divergence property of the PCAFs associated with symmetric Hunt processes 
(Theorem \ref{thm:pcaf-div}). 
By comparison with the previous results, 
we can allow the potential measures to be singular with respect to the reference measure, 
and the underlying processes to be generated by regular Dirichlet forms 
with bounded measurable coefficients. 
Moreover, by combining Theorem \ref{thm:pcaf-div} with the result of \cite{T25} (see also Theorem {\rm \ref{thm:takeda}}), 
we present a criterion for the Liouville property 
for Schr\"odinger operators (Theorem \ref{thm:liouville}). 
Since we assume that some bounded ${\cal L}$-harmonic function must be constant 
(see Assumption \ref{assum:const} and Remark \ref{rem:0-1} for details), 
we can regard Theorem \ref{thm:liouville} as a statement 
for the preservation  of the Liouville property under the measure perturbation.  

Our crucial assumption for Theorem \ref{thm:pcaf-div} is the condition \eqref{eq:decay}, 
the decay property of the potential measure $\mu$ weighted 
by the expected Feynman-Kac functional $g^{\mu}$ in \eqref{eq:gauge}. 
It is difficult in general to check this assumption; 
however, by using the result of Mizuta \cite{M77} on Riesz potentials 
(see \cite[Remark 2.7]{HB25} concerning \cite[Theorem 2]{M77}), 
we can verify the assumption under the rotation invariance of the potential measures and the coefficients. 
We formulated the condition \eqref{eq:decay} by following Ben Chrouda-Ben Fredj \cite{BCBF18}, 
which concerns semilinear equations with the fractional Laplace operator.

We make a comment on the proof of Theorem \ref{thm:pcaf-div}. 
As will be mentioned in \eqref{eq:triviality}, the divergence of the PCAF is equivalent to the triviality of $g^{\mu}$. 
Taking this equivalence into account, in Proposition \ref{prop:g-eq}, we first characterize  the function $g^{\mu}$ as a solution to some integral equation 
by using the strong Markov property and the optional projection property for PCAFs (\cite[Theorem 7.10 (p.~320) and Theorem 20.6 (p.~350)]{RW00}).  
We then follow an approach of Ben Chrouda-Ben Fredj \cite[Section 3.3]{BCBF18} 
to deduce the triviality of $g^{\mu}$ from the integral equation and \eqref{eq:decay}.

As an application of Theorems \ref{thm:takeda} and \ref{thm:liouville} with \cite{S15, SU14}, 
we obtain necessary and sufficient conditions for the Liouville property for some Schr\"odinger operators 
in terms of the decay rates of the potentials at infinity/boundary.  
We here give two such examples related to the fractional Laplace operators.  
\begin{itemize}
\item For $\alpha\in (1,2)$, 
let ${\mathbf M}=(\{X_t\}_{t\ge 0},\{P_x\}_{x\in {\mathbb R}^d})$ be 
a (transient) censored $\alpha$-stable process on a bounded domain $D\subset {\mathbb R}^d$ 
with smooth boundary, 
and let ${\cal L}$ be the associated generator.
For $r\in {\mathbb R}$, 
we define   $w(x)=\delta_D(x)^{-r}$ with $\delta_D(x)=\inf\{|x-y| \mid y\in {\partial D}\}$. 
Then the Schr\"odinger operator $-{\cal L}+w$ satisfies the Liouville property if and only if $r\ge \alpha$ 
(see Example \ref{exam:boundary} for details).

\item For $\alpha\in (1,2)$ and $d>\alpha$, 
let ${\mathbf M}=(\{X_t\}_{t\ge 0},\{P_x\}_{x\in {\mathbb R}^d})$ be 
a symmetric $\alpha$-stable-like process on ${\mathbb R}^d$ satisfying the rotation invariance property, 
and let ${\cal L}$ be the associated  generator.  
For $R>0$, let $\delta_R$ be the surface measure on $\{x\in {\mathbb R}^d \mid |x|=R\}$.
For $p>0$, we define the measure $\nu$ on ${\mathbb R}^d$ by 
\[
\nu({\rm d}y)=\sum_{n=1}^{\infty}n^{-p}\,\delta_{n^p}({\rm d}y).
\]
Then the Schr\"odinger operator $-{\cal L}+\nu$ satisfies the Liouville property if and only if $p\le (2-\alpha)^{-1}$
(see Example \ref{exam:singular} for details).  
\end{itemize}

The rest of this paper is organized as follows. 
In Section \ref{sect:prelim}, 
we first recall the notions related to the Dirichlet form theory. 
We then obtain the probabilistic expression of the potentials in \eqref{eq:density-2}. 
In Section \ref{sect:pcaf}, we establish a sufficient condition for the divergence of PCAFs at the terminal time 
(Theorem \ref{thm:pcaf-div}). 
In Section \ref{sect:liouville}, we first recall the notions of resurrected Dirichlet forms 
and the Liouville property for Schr\"odinger operators. 
We then obtain a criterion for  the Liouville property (Theorem \ref{thm:liouville}).
Section \ref{sect:exam} is devoted to applications of Theorems \ref{thm:takeda} and \ref{thm:liouville}  
to concrete Schr\"odinger operators. 
In Appendix \ref{sect:rotation}, 
we first introduce the rotation invariance properties of  Feller processes and Dirichlet forms on ${\mathbb R}^d$. 
Using this formulation, we next show a rotation invariance property of $g^{\mu}$ (Proposition \ref{prop:rot-pcaf-1}).  
We finally establish a convergence property of the Riesz potential of the measure $\mu$ 
weighted by the function $g^{\mu}$ (Corollary \ref{cor:pot-decay}). 
This property will be utilized in Section \ref{sect:exam} in order to verify the condition \eqref{eq:decay}. 

\section{Preliminaries}\label{sect:prelim}
We first recall notions on Dirichlet forms by following \cite{CF12, FOT11}.
Let $E$ be a locally compact separable metric space, 
and let $m$ be a positive Radon measure on $E$ with full support. 
We assume that $E$ is not compact, and write $E_{\Delta}=E\cup\{\Delta\}$ for the one point compactification of $E$. 
Let ${\mathbf M}=(\Omega,\{X_t\}_{t\ge 0}, \{P_x\}_{x\in E},{\cal F}, \{{\cal F}_t\}_{t\ge 0}, \{\theta_t\}_{t\ge 0}, \zeta)$ 
be an $m$-symmetric Hunt process on $E$ 
such that $\Delta$ is a cemetery point. 
Here $\zeta=\inf\{t>0 \mid X_t\in \Delta\}$ is the lifetime and $\theta_t$ is a shift operator of sample paths; 
$X_s\circ \theta_t(\omega)=X_{s+t}(\omega)$ for $s,t\ge 0$ and $\omega\in \Omega$. 
$\{{\cal F}_t\}_{t\ge 0}$ denotes the minimal completed admissible filtration (see, e.g., \cite[p.~386]{FOT11} for definition). 
We assume that the associated Dirichlet form $({\cal E},{\cal F})$ is regular on $L^2(E;m)$. 

Let $C_0(E)$ be the set of continuous functions on $E$ with compact support. 
We know by \cite[Theorem 3.2.1]{FOT11} that 
the regular Dirichlet form $({\cal E},{\cal F})$ on $L^2(E;m)$ has the Beurling-Deny  expression as follows: 
for any $u,v\in {\cal F}\cap C_0(E)$, 
\begin{equation}\label{eq:BD}
{\cal E}(u,v)={\cal E}^{(c)}(u,v)+\iint_{E\times E\setminus {\rm diag}}(u(x)-u(y))(v(x)-v(y))\,J({\rm d}x, {\rm d}y)
+\int_E u(x)v(x)\,k({\rm d}x),
\end{equation}
where
\begin{itemize}
\item ${\cal E}^{(c)}$ is a symmetric form with the domain ${\cal F}\cap C_0(E)$ and satisfies the strong local property. 
\item $J$ is a symmetric positive Radon measure on $E\times E$ off the diagonal set 
\[
{\rm diag}=\{(x,y)\in E\times E \mid x=y\}.
\]
\item $k$ is a positive Radon measure on $E$. 
\end{itemize}
The measures $J$ and $k$ are called the jumping measure and the killing measure, respectively, associated with ${\cal E}$.

Let ${\cal B}(E)$ be the set of Borel measurable functions on $E$. 
Let ${\cal B}_b(E)$ be the set of bounded functions in ${\cal B}(E)$, 
and let  ${\cal B}^+(E)$ be the set of non-negative functions in ${\cal B}(E)$. 
We define the semigroup $\{p_t\}_{t\ge 0}$ and resolvent $\{G_{\alpha}\}_{\alpha>0}$ of ${\bf M}$, 
respectively, by
\[
p_tf(x)=E_x\left[f(X_t);t<\zeta\right], \quad t>0, \ x\in E, \ f\in {\cal B}(E)
\] 
and 
\[
G_{\alpha}f(x)=\int_0^{\infty}e^{-\alpha t}p_t f(x)\,{\rm d}t, \quad \alpha>0, \ x\in E, \ f\in {\cal B}(E)
\] 
provided that the right hand side of each equality above makes sense.
We impose the next assumption on $({\cal E},{\cal F})$ and $\{p_t\}_{t\ge 0}$. 
\begin{assum}\label{assum:p_t}\rm 
The next conditions are fulfilled.
\begin{enumerate}
\item[(i)] (No killing measure) \ 
No killing measure is present in the Beurling-Deny expression of $({\cal E},{\cal F})$, 
i.e., $k=0$.

\item[(ii)] (Transience) \
There exists $g\in L^1(E;m)\cap {\cal B}(E)$ such that 
$g>0$, $m$-a.e.\ 
and $\int_0^{\infty}p_tg\,{\rm d}t<\infty$, $m$-a.e.

\item[(iii)] (Irreducibility) \
If a Borel set $A\subset E$ is invariant with respect to $\{p_t\}_{t\ge 0}$, 
i.e., $p_t({\bf 1}_Af)={\bf 1}_A p_tf$ $m$-a.e.\ for any $f\in L^2(E;m)\cap {\cal B}_b(E)$ and $t>0$, 
then $m(A)=0$ or $m(E\setminus A)=0$.  

\item[(iv)] (Strong Feller property) \
For any $f\in {\cal B}_b(E)$ and $t>0$, $p_tf$ is a bounded and continuous function on $E$.  
\end{enumerate}
\end{assum}

By the strong Feller property, the semigroup $\{p_t\}_{t\ge 0}$ satisfies the absolute continuity condition:
\begin{itemize}
\item For each $t>0$ and $x\in E$, 
the transition function $p_t(x,\cdot)$ is absolutely continuous with respect to $m$, 
$p_t(x,{\rm d}y)=p_t(x,y)\,m({\rm d}y)$.
\end{itemize}
Then for each $\alpha\ge 0$, 
there exists a non-negative jointly measurable function $G_{\alpha}(x,y) \ (x,y\in E)$, 
which is $\alpha$-excessive in $x$ and $y$, such that $G_{\alpha}(x,y)=G_{\alpha}(y,x)$ for any $x,y\in E$ and  
\[
G_{\alpha}f(x)=E_x\left[\int_0^{\zeta}e^{-\alpha t}f(X_t)\,{\rm d}t\right]
=\int_E G_{\alpha}(x,y)f(y)\,m({\rm d}y), \quad f\in {\cal B}^+(E)
\]
(\cite[Lemma 4.2.4]{FOT11}). 
In particular, we have $G_{\alpha}(x,y)=\int_0^{\infty} e^{-\alpha t}p_t(x,y)\,{\rm d}t$. 
We simply write $G(x,y)$ for $G_0(x,y)$. 

We say that a positive Radon measure $\mu$ on $E$ is of finite energy integral ($\mu\in S_0$ in notation) if  
\[
\iint_{E\times E}G_1(x,y)\,\mu({\rm d}x)\mu({\rm d}y)<\infty.
\]
We then define  
\[
S_{00}=\left\{\mu\in S_0 \mid \mu(E)<\infty, \ \sup_{x\in {\mathbb R}^d}G_1\mu(x)<\infty\right\},
\]
where $G_{\alpha}\mu(x)=\int_E G_{\alpha}(x,y)\,\mu({\rm d}y)$ is the $\alpha$-potential of $\mu$. 
See \cite[Exercise 4.2.2]{FOT11} for the equivalence of the definitions of $S_0$ and $S_{00}$ above 
and those in \cite[(2.2.1), (2.2.10)]{FOT11}. 
A positive Borel measure $\mu$ on $E$ is said to be smooth in the strict sense ($\mu\in S_1$ in notation)
if there exists a sequence $\{B_n\}_{n\ge 1}$ of Borel sets increasing to $E$ such that 
${\bf 1}_{B_n}\cdot\mu\in S_{00}$ for each $n\ge 1$ and 
\[
P_x\left(\lim_{n\rightarrow\infty}\sigma_{E\setminus B_n}\ge \zeta\right)=1, \quad x\in E.
\]
Here for a Borel set $B\subset E$, $\sigma_B=\inf\{t>0 \mid X_t\in B\}$ is the hitting time of ${\mathbf M}$ to $B$. 
By definition, $S_{00}\subset S_1$.

We next discuss the probabilistic expression of the potential of a measure in $S_1$.
Let  ${\mathbf A}_{c,1}^+$ be the family of positive continuous additive functionals in the strict sense with respect to ${\mathbf M}$
(see \cite[p.~222 and p.~235--236 for definition]{FOT11}). 
Then there exists a one to one correspondence (the so-called Revuz correspondence) 
between $S_1$ and ${\mathbf A}_{c,1}^+$ as follows: 
if $\mu\in S_1$ and $\{A_t\}_{t\ge 0}\in {\mathbf A}_{c,1}^+$ are in the Revuz correspondence, 
then for any $\alpha\ge 0$ and $f,h\in {\cal B}^+(E)$, 
\begin{equation}\label{eq:Revuz}
\int_E E_x\left[\int_0^{\infty}e^{-\alpha t}f(X_t)\,{\rm d}A_t\right]h(x)\,m({\rm d}x)
=\int_E G_{\alpha}h(x)f(x)\,\mu({\rm d}x)
\end{equation}
(\cite[Theorem 5.1.3]{FOT11}). 
In particular, if $\mu\in S_{00}$,  then by \cite[Theorem 5.1.6]{FOT11}, 
\[
E_x\left[\int_0^{\infty}e^{-t}\,{\rm d}A_t\right]=G_1\mu(x), \quad x\in E.
\]

By \eqref{eq:Revuz}, we have 
\begin{lem}
Let $\mu\in S_1$ and $\{A_t^{\mu}\}_{t\ge 0}\in {\mathbf A}_{c,1}^+$ be in the Revuz correspondence. 
Then for any $f\in {\cal B}^+(E)$, 
\begin{equation}\label{eq:density-2}
E_x\left[\int_0^{\zeta}f(X_s)\,{\rm d}A_s^{\mu}\right]=\int_E G(x,y)f(y)\,\mu({\rm d}y), \quad x\in E.
\end{equation}
\end{lem}

\begin{proof}
Let $f\in {\cal B}^+(E)$. 
Then by \eqref{eq:Revuz}, 
\begin{equation}\label{eq:density}
E_x\left[\int_0^{\zeta}f(X_s)\,{\rm d}A_s^{\mu}\right]=\int_E G(x,y)f(y)\,\mu({\rm d}y), \quad \text{$m$-a.e.\ $x\in E$}.
\end{equation}
Hence for any $x\in E$ and $t>0$, we have by the Fubini theorem,
\begin{equation}\label{eq:density-1}
\begin{split}
\int_E p_t(x,z)E_z\left[\int_0^{\zeta}f(X_s)\,{\rm d}A_s^{\mu}\right]\,m({\rm d}z)
&=\int_E p_t(x,z) \left(\int_E G(z,y)f(y)\,\mu({\rm d}y)\right)\,m({\rm d}z)\\
&=\int_E \left(\int_E p_t(x,z) G(z,y)\, m({\rm d}z) \right)f(y)\,\mu({\rm d}y).
\end{split}
\end{equation}

By the Markov property, we have  for any $x\in E$,
\begin{equation*}
\begin{split}
&\int_E p_t(x,z)E_z\left[\int_0^{\zeta}f(X_s)\,{\rm d}A_s^{\mu}\right]\,m({\rm d}z)
=E_x\left[E_{X_t}\left[\int_0^{\zeta}f(X_s)\,{\rm d}A_s^{\mu}\right];t<\zeta\right]\\
&=E_x\left[\int_t^{\zeta}f(X_s)\,{\rm d}A_s^{\mu};t<\zeta\right]
\rightarrow E_x\left[\int_0^{\zeta}f(X_s)\,{\rm d}A_s^{\mu}\right] \quad (t\rightarrow +0).
\end{split}
\end{equation*}
On the other hand, 
since $G(x,y)=\int_0^{\infty}p_t(x,y)\,{\rm d}t$  and the semigroup property of $\{p_t\}_{t\ge 0}$ yields 
\[
\int_E p_t(x,z)p_s(z,y)\,m({\rm d}z)=p_{t+s}(x,y),
\]
the Fubini theorem implies that for any $x,y\in E$, 
\[
\int_E p_t(x,z)G(z,y)\,m({\rm d}z)=\int_t^{\infty}p_s(x,y)\,{\rm d}s \rightarrow G(x,y) \quad (t\rightarrow +0).
\]
Therefore, by letting $t\rightarrow +0$ in \eqref{eq:density-1}, 
we obtain \eqref{eq:density-2}. 
\end{proof}

\section{Divergence of positive continuous additive functionals}\label{sect:pcaf}
We keep the same notations and setting as in the previous section. 
In this section, we follow the argument of \cite[Section 3.2]{BCBF18}  
to show a sufficient condition for the almost sure divergence of positive continuous additive functionals 
associated with smooth measures in the strict sense (Theorem \ref{thm:pcaf-div}). 

Let $\mu\in S_1$,  
and let $\{A_t^{\mu}\}_{t\ge 0}\in {\mathbf A}_{c,1}^+$ be in the Revuz correspondence to $\mu$.
Define the function $g^{\mu}$ on $E$ by 
\begin{equation}\label{eq:gauge}
g^{\mu}(x)=E_x\left[e^{-A_{\zeta}^{\mu}}\right], \quad x\in E,
\end{equation}
which is a probabilistic expression of the Liouville function in \cite[p.~668]{GH98}. 
Noting that  
\begin{equation}\label{eq:triviality}
P_x(A_{\zeta}^{\mu}=\infty)=1 \quad (x\in E) \iff g^{\mu}(x)=0 \quad (x\in E),
\end{equation}
we characterize $g^{\mu}$ as a solution to some equation:

\begin{prop}\label{prop:g-eq}
For any $x\in E$, 
\[
g^{\mu}(x)+\int_E G(x,y)g^{\mu}(y)\,\mu({\rm d}y)=P_x(A_{\zeta}^{\mu}<\infty).
\]
\end{prop}

\begin{proof}
We see from \eqref{eq:density-2} and \cite[Theorem 7.10 (p.~320) and Theorem 20.6 (p.~350)]{RW00} that
\begin{equation}\label{eq:optional}
\begin{split}
\int_E G(x,y)g^{\mu}(y)\,\mu({\rm d}y)
=E_{x}\left[\int_0^\zeta g^\mu(X_t)\,{\rm d}A_t^{\mu}\right]
&=E_{x}\left[\int_0^{\infty}1_{\{t<\zeta\}}E_{X_t}\left[e^{-A_\zeta^{\mu}}\right]\,{\rm d}A_t^{\mu}\right]\\
&=E_{x}\left[\int_0^{\infty}E_x\left[1_{\{t<\zeta\}}e^{-A_\zeta^{\mu}\circ\theta_t}|{\cal F}_t\right]\,{\rm d}A_t^{\mu}\right].
\end{split}
\end{equation}
Since 
\[
A_\zeta^{\mu}=A^\mu_t+A_\zeta^{\mu}\circ \theta_t\ \ \text{on}\ \{t<\zeta\},
\]
the last expression of \eqref{eq:optional} equals 
\begin{equation*}
\begin{split}
E_x\left[e^{-A^\mu_\zeta}\int_0^\zeta e^{A_t^{\mu}}\,{\rm d}A_t^{\mu};A^\mu_\zeta<\infty\right]
&=E_x\left[1-e^{-A_\zeta^{\mu}};A^\mu_\zeta<\infty\right]\\
&=P_x(A^\mu_\zeta<\infty)-E_x\left[e^{-A_\zeta^{\mu}}\right].
\end{split}
\end{equation*}
The proof is complete.
\end{proof}

To establish our result below, we need to impose the next condition on $A_t^{\mu}$.
\begin{assum}\label{assum:const}
${\mathbf M}$ is conservative and the function $h^{\mu}(x)=P_x(A_{\infty}^{\mu}<\infty) \ (x\in E)$ is constant. 
\end{assum}

\begin{rem}\label{rem:0-1}\rm 
If ${\mathbf M}$ is conservative,  then by the Markov property, 
\begin{equation*}
\begin{split}
E_x\left[h^{\mu}(X_t)\right]
&=E_x\left[P_{X_t}\left(A_{\infty}^{\mu}<\infty \right)\right]
=E_x\left[P_x\left(A_{\infty}^{\mu}\circ \theta_t<\infty \mid {\cal F}_t\right)\right]\\
&=P_x(A_{\infty}^{\mu}\circ\theta_t<\infty)=P_x(A_{\infty}^{\mu}<\infty)=h^{\mu}(x), \quad x\in E.
\end{split}
\end{equation*}
Namely, $p_th^{\mu}=h^{\mu}$ for any $t\ge 0$ and so $h^{\mu}$ is continuous by the strong Feller property. 
\begin{enumerate}
\item[(i)]
Let ${\mathbf M}$ be a symmetric L\'evy process on ${\mathbb R}^d$. 
Then ${\mathbf M}$ is conservative and the associated Dirichlet form is regular on $L^2({\mathbb R}^d;{\rm d}x)$ (\cite[Examples 1.4.1 and 4.1.1]{FOT11}).
If ${\mathbf M}$ satisfies the strong Feller property, 
and if the density of the transition function $q_t(x) \ (x\in {\mathbb R}^d)$ is bounded and continuous for some $t>0$, 
then $h^{\mu}$ is constant by \cite[Proposition 3]{K21}. 
Therefore, Assumption \ref{assum:const} is fulfilled. 
\item[(ii)] Assume that ${\mathbf M}$ is conservative. 
Since $P_x(A_t^{\mu}<\infty, \ t\ge 0)=1$, 
we know that 
\[
\{A_{\infty}^{\mu}<\infty\}=\{A_{\infty}^{\mu}\circ\theta_t<\infty\}, \quad \text{$P_x$-a.s., \ $t\ge 0$}.
\]
Hence the event $\{A_{\infty}^{\mu}<\infty\}$ belongs to the tail $\sigma$-field 
\[
{\cal T}=\bigcap_{t>0}\sigma\left(X_s, s\ge t\right).
\]
In particular,  if the tail event is trivial, 
then $P_x(A_{\infty}^{\mu}<\infty)=1$ for any $x\in E$, or $P_x(A_{\infty}^{\mu}<\infty)=0$ for any $x\in E$, 
whence Assumption \ref{assum:const} is fulfilled. 
Such a tail triviality is established for symmetric diffusion processes on ${\mathbb R}^d$ 
generated by regular Dirichlet forms \eqref{eq:form-diffusion} below
(\cite[Proposition 2.3]{BK00}), 
and for symmetric stable-like processes on ${\mathbb R}^d$ 
generated by regular Dirichlet forms \eqref{eq:form-stable-like} below (\cite[Theorem 2.10]{KKW17}). 
\end{enumerate}
\end{rem}

Using Proposition \ref{prop:g-eq}, we  prove
\begin{thm}\label{thm:pcaf-div}
Assume that Assumptions {\rm \ref{assum:p_t}} and {\rm \ref{assum:const}} are satisfied, and 
\begin{equation}\label{eq:decay}
\lim_{x\rightarrow \Delta}\int_E G(x,y)g^{\mu}(y)\,\mu({\rm d}y)=0.
\end{equation}
If there exists $x_0\in E$ such that $\int_E G(x_0,y)\,\mu({\rm d}y)=\infty$ 
and $\int_K G(x_0,y)\,\mu({\rm d}y)<\infty$ for any compact set $K$ in $E$,
then $P_x(A_{\infty}^{\mu}=\infty)=1$ for any $x\in E$.
\end{thm}

\begin{proof}
Let $h^{\mu}(x)=P_x(A_{\infty}^{\mu}<\infty)$. 
Since $h^{\mu}$ is constant by Assumption \ref{assum:const},  
we can set $h^{\mu}(x)=c \ (x\in E)$ for some $c\in [0,1]$. 
Then by Proposition \ref{prop:g-eq}, 
\[
g^{\mu}(x)+\int_E G(x,y)g^{\mu}(y)\,\mu({\rm d}y)=c, \quad x\in E. 
\]
Letting $x\rightarrow \Delta$, we have by \eqref{eq:decay},
\[
\lim_{x\rightarrow \Delta}g^{\mu}(x)=c. 
\]
Therefore,  there exists a compact set $K$ in $E$ such that for any $y\in E\setminus K$, 
$g^{\mu}(y)\ge c/2$,
which implies that 
\[
c=g^{\mu}(x_0)+\int_E G(x_0,y)g^{\mu}(y)\,\mu({\rm d}y)
\ge \int_{E\setminus K} G(x_0,y)g^{\mu}(y)\,\mu({\rm d}y)
\ge \frac{c}{2}\int_{E\setminus K} G(x_0,y)\,\mu({\rm d}y).
\]
Since $\int_{E\setminus K} G(x_0,y)\,\mu({\rm d}y)=\infty$ by assumption, 
we have $c=0$ and so the proof is complete.
\end{proof}

\section{Resurrected Dirichlet forms and Liouville property for Schr\"odinger operators}\label{sect:liouville}
In this section, we apply Theorem \ref{thm:pcaf-div} to the Liouville property for Schr\"odinger operators. 
Let $C_0(E)$ be the set of continuous functions on $E$ with compact support. 
Let $({\cal E},{\cal F})$ be a regular Dirichlet form on $L^2(E;m)$. 
Since each $u\in {\cal F}$ admits its quasi continuous $m$-version (\cite[Theorem 2.1.3]{FOT11}),  
we may and do assume that $u\in {\cal F}$ is already quasi continuous.

Following \cite{T25}, we first introduce the notion of solutions to the Laplace type equation associated with $({\cal E},{\cal F})$.  
Let ${\cal F}_{{\rm loc}}$ be the set of functions $u$ on $E$ such that, 
for any relatively compact open set $G$ in $E$, 
there exists $u_G\in {\cal F}$ such that $u=u_G$, $m$-a.e.\ on $G$. 
Then for the jumping measure $J({\rm d}x,{\rm d}y)$ in \eqref{eq:BD} and for any $u\in {\cal F}_{{\rm loc}}$, 
we define the Borel measure $\mu_{\langle u\rangle}^{j}$ on $E$ by 
\[
\mu_{\langle u\rangle}^{j}(B)=\iint_{B\times E\setminus {\rm diag}}(u(x)-u(y))^2\,J({\rm d}x, {\rm d}y), \quad B\in {\cal B}(E).
\]
We further define a subspace ${\cal F}_{{\rm loc}}^{\dagger}$ of ${\cal F}_{{\rm loc}}$ by following  \cite{FKW14, K10}:
\[
{\cal F}_{{\rm loc}}^{\dagger}=\{u\in {\cal F}_{{\rm loc}} \mid \text{$\mu_{\langle u\rangle}^j$ is a Radon measure on $E$}\}.
\]
Let  ${\cal L}$ be the generator associated with $({\cal E},{\cal F})$. 
We say that a function $h\in {\cal F}_{{\rm loc}}^{\dagger}\cap L^{\infty}(E;m)$ is a solution to ${\cal L}u=0$ if 
${\cal E}(h,\varphi)=0$ for any functions $\varphi\in {\cal F}\cap C_0(E)$. We write ${\cal S}({\cal E})$ for the set of
bounded solutions. 
We say that the operator ${\cal L}$ satisfies the {\it Liouville property} if $h=0$ q.e.\ for any 
$h\in{\cal S}({\cal E})$, where q.e.\ is an abbreviation for quasi-everywhere (see \cite[Section 2.1]{FOT11} for definition).

We next discuss the resurrection of $({\cal E},{\cal F})$. 
Let $({\cal F}_e,{\cal E})$ be the extended Dirichlet space of $({\cal F},{\cal E})$ (see \cite[p.~41]{FOT11} for definition). 
Then by \cite[Theorem 2.1.7]{FOT11}, any $u\in {\cal F}_e$ admits a quasi-continuous modification, which is still expressed as $u$. 
In particular, the energy form ${\cal E}$ on ${\cal F}_e\times {\cal F}_e$ 
also admits a unique Beurling-Deny representation \eqref{eq:BD} 
(\cite[Theorem 4.3.3 (i)]{CF12}).
For $u,v\in {\cal F}_e$, we define 
\[
{\cal E}^{\rm res}(u,v)={\cal E}^{(c)}(u,v)+\iint_{E\times E\setminus {\rm diag}}(u(x)-u(y))(v(x)-v(y))\,J({\rm d}x, {\rm d}y)
\]
and $m_k=m+k$.
Then by \cite[(5.2.24)]{CF12}, $({\cal E}^{\rm res},{\cal F})$ is also a regular Dirichlet form on $L^2(E;m_k)$.
Moreover, if  $({\cal F}_e^{\rm res}, {\cal E}^{\rm res})$ denotes 
the extended Dirichlet space of $({\cal F}, {\cal E}^{\rm res})$, 
then ${\cal F}_e={\cal F}_e^{\rm res}\cap L^2(E;k) \subset {\cal F}_e^{\rm res}$ (\cite[p.~188]{CF12}). 
We also define 
\[
{\cal F}^{\rm res}={\cal F}_e^{\rm res}\cap L^2(E;m).
\]
Then by \cite[Theorem 5.2.17]{CF12}, 
$({\cal E}^{\rm res}, {\cal F}^{\rm res})$ is a regular Dirichlet form on $L^2(E;m)$, 
which is called the resurrected Dirichlet form of $({\cal E},{\cal F})$.
In particular, ${\cal F}\cap C_0(E)$ is a special standard core of $({\cal E}^{\rm res}, {\cal F}^{\rm res})$. 
Therefore, if we define ${\cal E}_1^{{\rm res}}(u,u)={\cal E}^{{\rm res}}(u,u)+\int_E u^2\,{\rm d}m$ for $u\in {\cal F}$, 
then ${\cal F}^{\rm res}$ is a $\sqrt{{\cal E}_1^{{\rm res}}}$-closure of ${\cal F}\cap C_0(E)$.

Let us apply the argument above to perturbed Dirichlet forms. 
For the rest of this section, we assume that $({\cal E},{\cal F})$ has no killing measure, that is, $k=0$. 
Let $\mu$ be a positive Radon measure on $E$ charging no set of zero capacity.
Let $({\cal E}^{\mu},{\cal F}^{\mu})$ be a perturbed Dirichlet form on $L^2(E;m)$ defined by 
\[
{\cal F}^{\mu}={\cal F}\cap L^2(E;\mu), \quad {\cal E}^{\mu}(u,v)={\cal E}(u,v)+\int_E uv\,{\rm d}\mu, \quad  u,v\in {\cal F}^{\mu}
\]
(see, e.g., \cite[Section 5.1]{CF12} or \cite[Section 6.1]{FOT11}). 
Then $({\cal E}^{\mu},{\cal F}^{\mu})$ is a regular Dirichlet form on $L^2(E;m)$ such that 
${\cal F}\cap C_0(E)$ is a special standard core 
(see \cite[Theorem 5.1.6]{CF12} or  \cite[Theorem 6.1.2]{FOT11}).  
In particular, $({\cal E}^{\mu},{\cal F}^{\mu})$ has the killing measure $\mu$. 
We define the spaces $({\cal F}_{{\rm loc}}^{\mu})^\dagger$ and ${\cal S}({\cal E}^\mu)$ 
similarly to ${\cal F}_{{\rm loc}}^{\dagger}$ and ${\cal S}({\cal E})$, respectively.

Let  $({\cal F}^{\mu}_e,{\cal E}^{\mu})$ be 
the extended Dirichlet space of $({\cal F}^{\mu},{\cal E}^{\mu})$. 
Then by \cite[Proposition 5.1.9]{CF12}, we know that 
${\cal F}^{\mu}_e={\cal F}_e\cap L^2(E;\mu)$.
Moreover, if $(({\cal E}^{\mu})^{{\rm res}},({\cal F}^{\mu})^{{\rm res}})$ denotes 
the resurrected Dirichlet form of $({\cal E}^{\mu},{\cal F}^{\mu})$, 
then $({\cal E}^{\mu})^{\rm res}(u,u)={\cal E}(u,u)$ for any $u\in {\cal F}\cap C_0(E)$  and so 
\begin{equation}\label{eq:identify}
(({\cal E}^{\mu})^{\rm res},({\cal F}^{\mu})^{\rm res})=({\cal E}, {\cal F}).
\end{equation}

We finally establish a criterion for the Liouville property for Schr\"odinger operators. 
Let ${\cal L}^{\mu}$ be the generator of the perturbed Dirichlet form $({\cal E}^{\mu},{\cal F}^{\mu})$ on $L^2(E;m)$.
As in the previous two sections, 
let ${\mathbf M}$ be an $m$-symmetric Hunt process on $E$ such that 
the associated Dirichlet form $({\cal E},{\cal F})$ is regular on $L^2(E;m)$. 
Assume that ${\mathbf M}$ satisfies Assumption \ref{assum:p_t}. 
Let $\mu$ be  a positive Radon measure on $E$ belonging to  $S_1$, 
and  let $\{P_x^{\mu}\}_{x\in {\mathbb R}^d}$ be the law of the subprocess ${\mathbf M}^{\mu}$ of ${\mathbf M}$ 
generated by the multiplicative additive functional $e^{-A_t^{\mu}}$. 
Then ${\mathbf M}^{\mu}$ is associated with $({\cal E}^{\mu},{\cal F}^{\mu})$. 
Let $\zeta^{\mu}$ be the lifetime of ${\mathbf M}^{\mu}$, and $(\zeta^{\mu})^{i}$ its totally inaccessible part
(see \cite[(2.17)]{T25} for definition). 
Then \eqref{eq:identify} implies that 
an $m$-symmetric Hunt process generated by $(({\cal E}^{\mu})^{\rm res},({\cal F}^{\mu})^{\rm res})$ 
is identified with ${\mathbf M}$, and so 
\begin{equation}\label{eq:lifetime}
P_x^{\mu}(\zeta^\mu=(\zeta^\mu)^i<\infty)=1-E_x\left[e^{-A_{\infty}^{\mu}}\right]
\end{equation}
(\cite[Corollary 2.9 (ii)]{T25}). 
In particular, by \cite[Theorem 2.10, Corollary 2.21]{T25}, we obtain 

\begin{thm}{\rm (\cite[Theorem 2.10, Corollary 2.21]{T25})}\label{thm:takeda} 
Assume that ${\mathbf M}$ satisfies Assumption {\rm \ref{assum:p_t}}. 
Let $\mu$ be a positive Radon measure on $E$ belonging to $S_1$. 
Then the next statements are equivalent{\rm :}
\begin{enumerate}
\item[{\rm (i)}]
$P_x^{\mu}(\zeta^\mu=(\zeta^\mu)^i<\infty)=1, \ x\in E.$
\item[{\rm (ii)}] 
$P_x(A^\mu_\infty=\infty)=1, \ x\in E.$
\item[{\rm (iii)}] ${\cal L}^\mu\ \text{satisfies the Liouville property}$.
\end{enumerate}
\end{thm}

Hence by combining Theorem \ref{thm:pcaf-div} with Theorem \ref{thm:takeda}, 
we have 
\begin{thm}\label{thm:liouville}
Assume the same conditions as in Theorem {\rm \ref{thm:takeda}}.
\begin{enumerate}
\item[{\rm (i)}]
The operator ${\cal L}^\mu\ \text{satisfies the Liouville property}$ under the same setting as in Theorem {\rm \ref{thm:pcaf-div}}.
\item[{\rm (ii)}]
Assume that for some $x_0\in E$, $\int_E G(x_0,y)\,\mu({\rm d}y)<\infty$. 
Then there exists a solution $h\in ({\cal F}^{\mu}_{\rm loc})^{\dagger}\cap L_{{\rm loc}}^{\infty}(E;m)$ to ${\cal L}^{\mu}h=0$ 
such that the set $\{x\in E \mid h(x)\ne 0\}$ has positive capacity with respect to $({\cal E},{\cal F})$.
\end{enumerate}
\end{thm}

We refer to \cite[Section 2.1]{FOT11} for the definition of the capacity with respect to $({\cal E},{\cal F})$.

\begin{rem}\rm
Following \cite[Definition 1.1]{GH98}, we say that a smooth measure $\mu\in S_1$ is {\it big} ({\it non-big} otherwise) 
if the operator ${\cal L}^\mu$ satisfies the Liouville property.
It follows from Theorem \ref{thm:takeda} that $\mu\in S_1$ is {big} if and only if 
$$
g^\mu(x)\equiv 0\ \left(\Longleftrightarrow P_x(A^\mu_\zeta=\infty)=1\ \ \text{for all } x\in E\right).
$$
If $\mu$ is supported on a {\it non-thick set} $B$ (i.e., there exists  $x_0\in E$ such that $P_{x_0}(\sigma_B<\zeta)<1$),
then $g^\mu(x_0)>0$ because $P_{x_0}(A^\mu_\zeta=0)>0$ on account of $P_{x_0}(\sigma_B\ge\zeta)>0$. 
Hence we see that if $\mu$ is supported on a non-thick set, then $\mu$ is not big (cf. \cite[Theorem 1.1]{GH98}).
\end{rem}

\section{Examples}\label{sect:exam}
In this section, we present necessary and sufficient conditions for the Liouville property 
for some Schr\"odinger operators related to the (fractional) Laplace operator. 
In what follows, we write ${\rm d}x$ for the Lebesgue measure on ${\mathbb R}^d$. 

\subsection{Potentials of positive functions}
In this subsection, we take some positive functions as a potential term. 
We here use Theorem \ref{thm:takeda}.

\begin{exam}\label{exam:full-support}\rm 
Let us discuss how the decay rate of the potential affects the validity of the Liouville property. 
\begin{enumerate}
\item[(i)]
Let $\{a_{ij}(x)\}_{1\le i,j\le d}$ be a family of bounded Borel measurable functions on ${\mathbb R}^d$ such that 
$a_{ij}(x)=a_{ji}(x)$ for any $x\in {\mathbb R}^d$, and 
for some $c>0$, 
\[
\sum_{i,j=1}^da_{ij}(x)\xi_i\xi_j\ge c|\xi|^2, \quad x\in {\mathbb R}^d, \ \xi \in {\mathbb R}^d.
\]
We define the quadratic form $({\cal E},{\cal F})$ on $L^2({\mathbb R}^d;{\rm d}x)$ by 
\begin{equation}\label{eq:form-diffusion}
\begin{split}
{\cal F}&=\left\{u\in L^2({\mathbb R}^d;{\rm d}x) \mid \frac{\partial u}{\partial x_i}\in L^2({\mathbb R}^d;{\rm d}x) \ (1\le i\le d)\right\},\\
{\cal E}(u,v)&=\int_{{\mathbb R}^d} \sum_{i,j=1}^d a_{ij}(x)\frac{\partial u}{\partial x_i}(x)\frac{\partial v}{\partial x_j}(x)\,{\rm d}x, \quad u,v\in {\cal F}, 
\end{split}
\end{equation}
where we take derivatives in the Schwartz distribution sense. 
Then $({\cal E},{\cal F})$ is a regular Dirichlet form on $L^2({\mathbb R}^d;{\rm d}x)$ 
and there exists an associated symmetric diffusion process on ${\mathbb R}^d$. 
Since the density of the transition function is continuous and fulfills the Gaussian bound (\cite{A67, BK00}), 
we have a version of the process, say ${\mathbf M}$, having the Feller and strong Feller property. 
In particular, the semigroup of ${\mathbf M}$ satisfies Assumption \ref{assum:p_t}. 
We write ${\cal L}$ for the $L^2$-generator of  $({\cal E},{\cal F})$. 

For $p\in {\mathbb R}$, let $w(x)$ be a positive function on ${\mathbb R}^d$ such that for some positive constants $c_1$ and $c_2$, 
\begin{equation}\label{eq:w}
c_1(1+|x|)^p\le w(x)\le c_2(1+|x|)^p, \quad x\in {\mathbb R}^d.
\end{equation}
Let $\mu$ be a measure on ${\mathbb R}^d$ defined as $\mu({\rm d}x)=w(x)\,{\rm d}x$. 
Then $\mu$ is a positive Radon measure of full support, and we can show $\mu\in S_1$ in a similar way to \cite[Example 5.1.1]{FOT11}. 
If $\check{\bf M}$ denotes the time changed process of ${\mathbf M}$ 
with respect to the positive continuous additive functional $A_t^{\mu}=\int_0^t w(X_s)\,{\rm d}s$, 
then the lifetime of $\check{\mathbf M}$ is $A_{\infty}^{\mu}=\int_0^{\infty}w(X_s)\,{\rm d}s$.

We now assume that $d\ge 3$. 
Since ${\mathbf M}$ is transient and the Green function $G(x,y)$ satisfies 
$c_1|x-y|^{2-d}\le G(x,y)\le c_2|x-y|^{2-d}$ (\cite{A67}), 
we can follow the argument in \cite[Example 4.8]{S15} and \cite[Example 4.3]{SU14} to show that 
\[
P_x\left(\int_0^{\infty}w(X_s)\,{\rm d}s=\infty\right)=1 \quad (x\in {\mathbb R}^d)
\iff p\ge -2.
\]
By Theorem \ref{thm:takeda}, this condition is equivalent to the Liouville property for ${\cal L}-w$. 

\item[(ii)] 
Let $\alpha\in (0,2)$, and let $c(x,y)$ be a Borel measurable function on ${\mathbb R}^d\times {\mathbb R}^d$ such that 
for some positive constants $c_1$ and $c_2$, $c_1\le c(x,y)\le c_2$ for any $x,y\in {\mathbb R}^d$. 
We define the quadratic form $({\cal E},{\cal F})$ on $L^2({\mathbb R}^d;{\rm d}x)$ by 
\begin{equation}\label{eq:form-stable-like}
\begin{split}
{\cal F}
&=\left\{u\in L^2({\mathbb R}^d;{\rm d}x) \mid 
\iint_{{\mathbb R}^d\times {\mathbb R}^d\setminus {\rm diag}}\frac{(u(x)-u(y))^2}{|x-y|^{d+\alpha}}\,{\rm d}x{\rm d}y<\infty\right\},\\
{\cal E}(u,v)
&=\iint_{{\mathbb R}^d\times {\mathbb R}^d\setminus {\rm diag}}c(x,y)\frac{(u(x)-u(y))(v(x)-v(y))}{|x-y|^{d+\alpha}}\,{\rm d}x{\rm d}y, 
\quad u,v\in {\cal F},
\end{split}
\end{equation}
where ${\rm diag}=\{(x,y)\in {\mathbb R}^d\times {\mathbb R}^d \mid x=y\}$. 
Then $({\cal E},{\cal F})$ is a regular Dirichlet form on $L^2({\mathbb R}^d;{\rm d}x)$ 
and there exists an associated symmetric jump process on ${\mathbb R}^d$, 
which is called a symmetric stable-like process on ${\mathbb R}^d$ (\cite{CK03}).  
Since the density of the transition function is continuous and fulfills the stable-type bound (\cite{CK03, CK08, KKW17}), 
we have a version of the process, say ${\mathbf M}$, having the Feller and strong Feller property. 
In particular, the semigroup of ${\mathbf M}$ satisfies Assumption \ref{assum:p_t}. 
We write ${\cal L}$ for the $L^2$-generator of  $({\cal E},{\cal F})$.

We now assume that $d>\alpha$. 
Then ${\mathbf M}$ is transient and the Green function $G(x,y)$ satisfies 
$c_1|x-y|^{\alpha-d}\le G(x,y)\le c_2|x-y|^{\alpha-d}$ (\cite{CK03, CK08, KKW17}).
Let $p$ be a constant, 
and let $w(x)$ be a positive function on ${\mathbb R}^d$ satisfying \eqref{eq:w}. 
Then as in (i), we see by \cite[Example 4.8]{S15} and \cite[Example 4.3]{SU14} that 
\[
P_x\left(\int_0^{\infty}w(X_s)\,{\rm d}s=\infty\right)=1 \quad (x\in {\mathbb R}^d)
\iff p\ge -\alpha.
\]
By Theorem \ref{thm:takeda}, this condition is equivalent to the Liouville property for ${\cal L}-w$. 
\end{enumerate}
\end{exam}

\begin{exam}\label{exam:boundary}\rm 
We here discuss how the boundary behavior of the potential affects the validity of the Liouville property.
Let $D\subset {\mathbb R}^d$ be a bounded domain with smooth boundary $\partial D$, 
and let $C_0^{\infty}(D)$ be the set of smooth functions on $D$ with compact support.
\begin{enumerate}
\item[(i)]  
Let $d\ge 3$. We define the quadratic form $({\cal E},{\cal F})$ on $L^2(D;{\rm d}x)$ by 
\begin{equation*}
\begin{split}
{\cal D}({\cal E})
&=\left\{u\in L^2(D;{\rm d}x) \mid  \frac{\partial u}{\partial x_i}\in L^2(D;{\rm d}x) \ (1\le i\le d) \right\},\\
{\cal E}(u,v)&=\frac{1}{2}\int_D \sum_{i=1}^d \frac{\partial u}{\partial x_i}(x) \frac{\partial v}{\partial x_i}(x)\,{\rm d}x, \quad u,v\in {\cal D}({\cal E}).
\end{split}
\end{equation*}
Let  ${\cal F}$ be the closure of $C_0^{\infty}(D)$ with respect to the norm $\sqrt{{\cal E}_1}$ in ${\cal D}({\cal E})$, 
where ${\cal E}_1(u,u)={\cal E}(u,u)+\int_D u^2\,{\rm d}x$. 
Then $({\cal E},{\cal F})$ is a regular Dirichlet form on $L^2(D;{\rm d}x)$, 
and the corresponding symmetric Hunt process ${\mathbf M}$ is an absorbing Brownian motion on $D$. 
This process fulfills Assumption \ref{assum:p_t} (see, e.g., \cite[Theorem 2.2]{CZ95}). 
We write ${\cal L}$ for the $L^2$-generator of  $({\cal E},{\cal F})$. 

For $x\in D$, let $\delta_D(x)=\inf \{|x-y| \mid y\in \partial D\}$. 
For $r\in {\mathbb R}$, let $w(x)$ be a positive function on $D$ such that 
for some positive constants $c_1$ and $c_2$, 
\begin{equation}\label{eq:w-1}
\frac{c_1}{\delta_D(x)^r} \le w(x) \le \frac{c_2}{\delta_D(x)^r}, \quad x\in D.
\end{equation}
Let $\mu$ be a measure on $D$ defined by $\mu({\rm d}x)=w(x)\,{\rm d}x$, 
which is a positive Radon measure of full support and belongs to $S_1$. 
If $G_D(x,y)$ is the Green function of ${\mathbf M}$, then 
\[
G_D(x,y)\asymp \min\left\{\frac{1}{|x-y|^{d-2}}, \frac{\delta_D(x)^{1/2}\delta_D(y)^{1/2}}{|x-y|^d}\right\}, \quad x,y\in D
\]
(\cite[Theorems 1 and 3]{Z86}). 
Hence by following the argument of  \cite[Example 4.11 (ii)]{S15}, we have
\[
P_x\left(\int_0^{\tau_D}w(X_s)\,{\rm d}s=\infty\right)=1 \quad (x\in D) \iff r\ge -2.
\] 
By Theorem \ref{thm:takeda}, 
this condition is equivalent to the Liouville property for ${\cal L}-w$. 
\item[(ii)]
For $\alpha\in (0,2)$, 
we define the quadratic form $({\cal E},{\cal F})$ on $L^2(D;{\rm d}x)$ by 
\begin{equation*}
\begin{split}
{\cal D}({\cal E})
&=\left\{u\in L^2(D;{\rm d}x) \mid \iint_{D\times D\setminus {\rm diag}}\frac{(u(x)-u(y))^2}{|x-y|^{d+\alpha}}\,{\rm d}x{\rm d}y<\infty\right\},\\
{\cal E}(u,v)&=\iint_{D\times D\setminus {\rm diag}}\frac{(u(x)-u(y))(v(x)-v(y))}{|x-y|^{d+\alpha}}\,{\rm d}x{\rm d}y, \quad u,v\in {\cal D}({\cal E}).
\end{split}
\end{equation*}
Let ${\cal F}$ be the closure of $C_0^{\infty}(D)$ with respect to the norm $\sqrt{{\cal E}_1}$ in ${\cal D}({\cal E})$. 
Then $({\cal E},{\cal F})$ is a regular Dirichlet form on $L^2(D;{\rm d}x)$ 
and the associated symmetric Hunt process ${\mathbf M}$ is called a censored stable process on $D$ (\cite{BBC03}), 
which satisfies Assumption \ref{assum:p_t}. 
In fact, we can verify the strong Feller property for ${\mathbf M}$ 
by using  \cite[Theorem 1.1]{CKS10} and \cite[Proposition 2.14, Theorem 3.2 and Remark 3.4]{CKSV20}. 
We write ${\cal L}$ for the $L^2$-generator of  $({\cal E},{\cal F})$. 

Suppose first that $\alpha\in (0,1]$. 
Then the censored stable process is conservative (\cite[Theorem 1.1]{BBC03}), and thus recurrent. 
Hence for any non-negative function $w$ with ${\rm d}x(\{x\in D\mid w(x)>0\})>0$,
\[
P_x\left(\int_0^{\tau_D}w(X_s)\,{\rm d}s=\infty\right)=1 \quad (x\in D),
\] 
and the Liouville property for ${\cal L}-w$ follows from Theorem \ref{thm:takeda}.

Suppose next that $\alpha\in (1,2)$. 
Then the censored stable process hits the boundary $\partial D$ almost surely, 
and so this process is explosive (\cite[Theorem 1.1]{BBC03}). 
Let $w(x)$ be a positive function on $D$ satisfying \eqref{eq:w-1}. 
Then as in (i), we see by  \cite[Example 4.11 (ii)]{S15} that 
\[
P_x\left(\int_0^{\tau_D}w(X_s)\,{\rm d}s=\infty\right)=1 \quad (x\in D) \iff r\ge -\alpha.
\] 
By  Theorem \ref{thm:takeda}, 
this condition is equivalent to the Liouville property for ${\cal L}-w$.

\end{enumerate}
\end{exam}

\subsection{Potentials of non-fully supported functions and measures}
When the potential term is not of full support, we can not use the results of \cite{S15, SU14} as in the previous subsection. 
However, if we assume the rotation invariance properties for the potential terms 
and Dirichlet forms, 
then we can verify the conditions in Theorem \ref{thm:pcaf-div} 
by Remark \ref{rem:0-1} (ii) and Corollary \ref{cor:pot-decay} below.

Let $({\cal E},{\cal F})$ be a rotation invariant, transient and regular Dirichlet form on $L^2({\mathbb R}^d;{\rm d}x)$ 
in Examples \ref{exam:diffusion} or \ref{exam:stable-like} below. 
Then there exists an associated symmetric Hunt process on ${\mathbb R}^d$, say ${\mathbf M}$,  
which satisfies the Feller and strong Feller properties. 
In particular, ${\mathbf M}$ fulfills Assumptions \ref{assum:p_t} and \ref{assum:const}.
Let $G(x,y)=G^{(\alpha)}(x,y)$ be the Green function of ${\mathbf M}$, 
where we set $\alpha=2$ if we take the Dirichlet form in Example \ref{exam:diffusion}. 
Recall that there exist positive constants $c_1$ and $c_2$ such that 
$c_1|x-y|^{\alpha-d}\le G(x,y)\le c_2|x-y|^{\alpha-d} \ (x\ne y)$.

We note that all the measures in the examples below are positive Radon measures in $S_1$, and rotation invariant. 
Hence \eqref{eq:decay} is fulfilled, 
and Theorems \ref{thm:pcaf-div} and \ref{thm:liouville} work well. 

\begin{exam}\label{exam:absolute}\rm 
Let $\{f(n)\}_{n\ge 1}$ be a positive sequence such that 
$\{f(n)\}_{n\ge 1}$ is strictly increasing, 
$\lim_{n\rightarrow\infty}f(n)=\infty$ and $\lim_{n\rightarrow\infty}f(n+1)/f(n)=1$. 
Let $\{h(n)\}_{n\ge 1}$ be a positive sequence so that 
\begin{equation}\label{eq:increasing}
f(n+1)-f(n)(1+h(n))\ge 0, \quad n\ge 1.
\end{equation}
Under this setting, if we define 
\[
a_n=f(n), \quad b_n=f(n)(1+h(n)), \quad n\ge 1,
\]
then $a_n<b_n\le a_{n+1}$ for any $n\ge 1$.

For $r\in {\mathbb R}$, let $\mu_r$ be a measure on ${\mathbb R}^d$ defined by 
\[
\mu_r({\rm d}y)=\sum_{n=1}^{\infty}{\bf 1}_{a_n\le |y|\le b_n}(y)|y|^{-r}\,{\rm d}y.
\]
Then the support of $\mu$ is an infinite and disjoint union of annuli. 
If $r>\alpha$, then $G\mu_r(0)<\infty$ and so $P_x(A_{\infty}^{\mu_r}<\infty)=1$ for any $x\in {\mathbb R}^d$; 
Theorem \ref{thm:liouville} implies that the Liouville property fails for ${\cal L}^{\mu}$.

In what follows, we assume that $r\le \alpha$. 
Then for any compact set $K$ in ${\mathbb R}^d$, 
there exists $c>0$ such that $\int_K G(0,y)\,\mu_r({\rm d}y)\le c\int_K G(0,y)\,{\rm d}y<\infty$. 
Moreover, 
\[
\int_{{\mathbb R}^d}G(0,y)\,\mu_r({\rm d}y)
\asymp \sum_{n=1}^{\infty}\int_{a_n\le |y|\le b_n}\frac{1}{|y|^{r+d-\alpha}}\,{\rm d}y
=
\begin{dcases}
\frac{\omega_d}{\alpha}\sum_{n=1}^{\infty}(b_n^{\alpha-r}-a_n^{\alpha-r}) & (r<\alpha),\\
\omega_d \log\frac{b_n}{a_n} & (r=\alpha),
\end{dcases}
\]
where $\omega_d$ is the surface area of a unit ball in ${\mathbb R}^d$. 
If $r<\alpha$, then 
\[
b_n^{\alpha-r}-a_n^{\alpha-r}
=f(n)^{\alpha-r}\{(1+h(n))^{\alpha-r}-1\}\sim (\alpha-r) f(n)^{\alpha-r}h(n) \quad (n\rightarrow\infty). 
\]
Since $\log(1+t)\sim t \ (t\rightarrow 0)$,  we also have
\[
\log\frac{b_n}{a_n}=\log(1+h(n))\sim h(n) \quad (n\rightarrow\infty).
\]
Therefore, Theorem \ref{thm:pcaf-div} implies that if  $r\le \alpha$ and $\alpha>1$, then 
\[
P_x(A_{\infty}^{\mu_r}=\infty)=1 \ (x\in {\mathbb R}^d) \iff \int_{{\mathbb R}^d}G(0,y)\,\mu_r({\rm d}y)=\infty \iff \sum_{n=1}^{\infty} f(n)^{\alpha-r}h(n)=\infty.
\]
By Theorem \ref{thm:liouville}, this condition is equivalent to the Liouville property for ${\cal L}^{\mu_r}$.
On the other hand, if  $r\le \alpha$ and $\alpha\le 1$, then 
\[
\sum_{n=1}^{\infty} f(n)^{\alpha-r}h(n)<\infty \iff \int_{{\mathbb R}^d}G(0,y)\,\mu_r({\rm d}y)<\infty 
\Longrightarrow P_x(A_{\infty}^{\mu_r}<\infty)=1 \ (x\in {\mathbb R}^d). 
\]

Assume now that $1<\alpha<2$ and $r\le \alpha$. 
Let us take $f(n)=n^p$ and $h(n)=n^{-q}$ for some positive constants $p$ and $q$ 
such that $q>1$, or $q=1$ and $p\ge 1$. 
We can see that this condition is equivalent to the validity of \eqref{eq:increasing}. 
Under this setting, we have $ f(n)^{\alpha-r}h(n)=n^{p(\alpha-r)-q}$ and so 
\[
P_x(A_{\infty}^{\mu_r}=\infty)=1 \ (x\in {\mathbb R}^d)  \iff q\le p(\alpha-r)+1.
\]
We assume in addition that $r=0$. 
Then the volume of the annulus $\{x\in {\mathbb R}^d \mid a_n\le |x|\le b_n\}$ is $\omega_d(b_n^d-a_n^d)$ 
and 
\[
\omega_d(b_n^d-a_n^d)=\omega_d n^{pd}\left\{\left(1+\frac{1}{n^q}\right)^d-1\right\}\sim d\omega_dn^{pd-q} \quad (n\rightarrow\infty). 
\]
Hence if $q>pd$, then the volume of the annulus decays but the Liouville property holds for ${\cal L}^{\mu}$.
\end{exam}

\begin{exam}\label{exam:singular}\rm 
Let $\alpha\in (1,2)$. 
For $R>0$, let $\delta_R$ be the surface measure on $\{x\in {\mathbb R}^d \mid |x|=R\}$.
Let $\{s_n\}$ be a positive increasing sequence such that $s_n\rightarrow\infty$ as $n\rightarrow\infty$. 
For $r\in {\mathbb R}$, if we define the measure 
\[
\nu({\rm d}y)=\sum_{n=1}^{\infty}|y|^{-r}\,\delta_{s_n}({\rm d}y)=\sum_{n=1}^{\infty}s_n^{-r}\,\delta_{s_n}({\rm d}y), 
\]
then as in Example \ref{exam:absolute}, we obtain 
\begin{equation}\label{eq:cond-1}
P_x(A_{\infty}^{\nu}=\infty)=1 \ (x\in {\mathbb R}^d) \iff \int_{{\mathbb R}^d}G(0,y)\,\nu({\rm d}y)=\infty \iff \sum_{n=1}^{\infty} s_n^{\alpha-1-r}=\infty.
\end{equation}
By Theorem \ref{thm:liouville}, this condition is equivalent to the Liouville property for ${\cal L}^{\nu}$.

We see by \eqref{eq:cond-1} that if $r\le \alpha-1$, then $P_x(A_{\infty}^{\nu}<\infty)=1$ for any $x\in {\mathbb R}^d$.
We now assume that $r>\alpha-1$. 
If we take $s_n=n^p$ for some $p>0$, then $s_n^{\alpha-1-r}=n^{p(\alpha-1-r)}$ and so 
\[
P_x(A_{\infty}^{\nu}=\infty)=1 \ (x\in {\mathbb R}^d) \iff p(\alpha-1-r)\ge -1\iff p\le \frac{1}{r-(\alpha-1)}.
\]
\end{exam}

\appendix
\section{Rotation invariance}\label{sect:rotation}

In this appendix, we first discuss the rotation invariance of the laws of Feller processes. 
We next introduce the rotation invariance of regular Dirichlet forms on $L^2({\mathbb R}^d;{\rm d}x)$. 
We finally discuss the distributions of positive continuous additive functionals in the strict sense 
for Feller processes associated with rotation invariant regular Dirichlet forms.  

\subsection{Rotation invariance of Feller processes}
For $x\in {\mathbb R}^d$, let $\delta_x$ be the Dirac measure at $x$. 
Let $p_t(x,A) \ (x\in {\mathbb R}^d, t\ge 0, A\in {\cal B}({\mathbb R}^d))$ be a sub-Markov transition function 
on $({\mathbb R}^d,{\cal B}({\mathbb R}^d))$ with $p_0(x,\cdot)=\delta_x(\cdot) \ (x\in {\mathbb R}^d)$. 
Let $\{p_t\}_{t\ge 0}$ be an associated Markovian semigroup defined by  
$p_tf(x)=\int_{{\mathbb R}^d}p_t(x,{\rm d}y)f(y)$ for $f\in {\cal B}_b({\mathbb R}^d)$. 
Let $C({\mathbb R}^d)$ be the set of continuous functions on ${\mathbb R}^d$, and let 
\[
C_{\infty}({\mathbb R}^d)
=\left\{f\in C({\mathbb R}^d) \mid \lim_{|x|\rightarrow\infty}f(x)=0\right\}. 
\]
We say that $\{p_t\}_{t\ge 0}$ is a Feller semigroup if  
$p_tf\in C_{\infty}({\mathbb R}^d)$ for any $t>0$ and $f\in C_{\infty}({\mathbb R}^d)$. 

As we see from \cite[Theorem 9.4 (p.~46--)]{BG68}, 
we can realize a Hunt process from the Feller semigroup as follows: 
Let $\Omega$ be 
the set of all functions $\omega: [0,\infty) \to ({\mathbb R}^d)_{\Delta}$ such that 
each $\omega$ is right continuous and has left limits in $[0,\infty)$, $\omega(\infty)=\Delta$,  
$\omega(s)=\Delta$ if $s\ge t$ and $\omega(t)=\Delta$.
We then define the function $X_t \ (t\ge 0)$ on $\Omega$ by $X_t(\omega)=\omega(t) \ (t\ge 0)$. 
For $t\ge 0$, we define the shift operator $\theta_t:\Omega \to\Omega$ by $\theta_t w(s)=w(t+s) \ (s\ge 0)$, 
and the lifetime $\zeta:\Omega\to [0, \infty]$ by $\zeta(\omega)=\inf\{t>0 \mid X_t(\omega)=\Delta\}$. If we further define
\[
{\cal F}_t^0=\sigma(X_s, \, s\le t) \  (t\ge 0), \quad {\cal F}^0=\sigma(X_s, \, s\ge 0), 
\]
then we can realize a Markov process 
\[
(\Omega, \{X_t\}_{t\ge 0},\{P_x\}_{x\in {\mathbb R}^d}, {\cal F}^0, \{{\cal F}_t^0\}_{t\ge 0},\{\theta_t\}_{t\ge 0},\zeta)
\] 
on $({\mathbb R}^d,{\cal B}({\mathbb R}^d))$ having $p_t(x,A)$ as the transition function.

Let ${\cal F}$ be a completion of ${\cal F}^0$ with respect to the family of measures 
$P_{\mu}(\cdot)=\int_{{\mathbb R}^d}\,\mu({\rm d}x)P_x(\cdot)$ for any finite measures $\mu$ on ${\mathbb R}^d$, 
and let ${\cal F}_t$ be a completion of ${\cal F}_t^0$ in ${\cal F}$. 
Then the filtration $\{{\cal F}_t\}_{t\ge 0}$ is right continuous and 
\begin{equation}\label{eq:realization}
{\mathbf M}=(\Omega, \{X_t\}_{t\ge 0},\{P_x\}_{x\in {\mathbb R}^d}, {\cal F}, \{{\cal F}_t\}_{t\ge 0},\{\theta_t\}_{t\ge 0},\zeta)
\end{equation}
is a Hunt process on $({\mathbb R}^d,{\cal B}({\mathbb R}^d))$ having $p_t(x,A)$ as the transition function. 
This process is called a Feller process. 

For a Feller process ${\mathbf M}$ in \eqref{eq:realization}, 
we can define the transformation of paths by orthogonal matrices:  
Let $O(d)$ denote the set of $d\times d$-orthogonal matrices. 
Then for $Q\in O(d)$, $w\in \Omega$ and $\Lambda\in {\cal F}$, 
we define $Qw\in \Omega$ and $Q\Lambda\in {\cal F}$, respectively, by 
$(Qw)(t)=Q(w(t)) \ (t\ge 0)$ and $Q\Lambda=\{Qw \mid w\in \Lambda\}$, 
where we set $Q\Delta=\Delta$. 
We can further define a probability measure $P_x^Q$ on $\Omega$ by 
\[
P_x^Q(\Lambda)=P_{Qx}(Q\Lambda), \quad \Lambda\in {\cal F}.
\]
If we assume in addition that ${\mathbf M}$ satisfies the strong Feller property, 
then for each $t>0$, there exists a mesurable function $(x,y)\mapsto p_t(x,y)$ on ${\mathbb R}^d\times {\mathbb R}^d$
such that $p_t(x,{\rm d}y)=p_t(x,y)\,{\rm d}y$. 
Then the rotation invariance of $p_t(x,y)$ yields that of $P_x$:
\begin{lem}\label{lem:rotation}
Let ${\mathbf M}$ be a Feller process on ${\mathbb R}^d$ as in \eqref{eq:realization}. 
Assume that ${\mathbf M}$ satisfies the strong Feller property and for any $Q\in O(d)$,
\[
p_t(x,y)=p_t(Qx,Qy), \quad x,y\in {\mathbb R}^d, \, t>0.
\]
Then for each $x\in {\mathbb R}^d$, $P_x(\Lambda)=P_x^Q(\Lambda)$ for any $\Lambda\in {\cal F}$.
\end{lem}
\begin{proof}
For any sequence $\{t_k\}_{k=1}^n\subset [0,\infty)$ with $t_1<\cdots<t_n$ 
and $A_1,\dots, A_n\in {\cal B}({\mathbb R}^d)$, 
we define a cylinder set ${\cal C}=\{X_{t_1}\in A_1, \dots, X_{t_n}\in A_n\}\in {\cal F}$. 
Since $Q{\cal C}=\{X_{t_1}\in QA_1, \dots, X_{t_n}\in QA_n\}\in {\cal F}$, 
we have by the Markov property and the change of variables formula with $z=Qy_1$,  
\begin{equation*}
\begin{split}
P_x^Q({\cal C})
&=P_{Qx}(X_{t_1}\in QA_1,\dots, X_{t_n}\in QA_n)\\
&=E_{Qx}[P_{X_{t_1}}(X_{t_2-t_1}\in QA_2, \dots, X_{t_n-t_1}\in QA_n); X_{t_1}\in QA_1]\\
&=\int_{QA_1}p_{t_1}(Qx,z)P_z(X_{t_2-t_1}\in QA_2, \dots, X_{t_n-t_1}\in QA_n)\,{\rm d}z\\
&=\int_{A_1}p_{t_1}(Qx,Qy_1)P_{Qy_1}(X_{t_2-t_1}\in QA_2, \dots, X_{t_n-t_1}\in QA_n)\,{\rm d}y_1.
\end{split}
\end{equation*}
Then by assumption and the inductive argument with the Markov property, 
the last term above is equal to 
\begin{equation*}
\begin{split}
&\int_{A_1}p_{t_1}(x,y_1)P_{Qy_1}(X_{t_2-t_1}\in QA_2, \dots, X_{t_n-t_1}\in QA_n)\,{\rm d}y_1\\
&=\int_{A_1}p_{t_1}(x,y_1)
\left\{\int_{A_2}p_{t_2-t_1}(y_1,y_2)\cdots\left(\int_{A_n}p_{t_n-t_{n-1}}(y_{n-1},y_n)\,{\rm d}y_n\right)\cdots{\rm d}y_2\right\}{\rm d}y_1\\
&=P_x(X_{t_1}\in A_1, \dots, X_{t_n}\in A_n)=P_x({\cal C}).
\end{split}
\end{equation*}
We thus have $P_x=P_x^Q$ on cylindrical sets, and so  $P_x=P_x^Q$ on ${\cal F}$. 
\end{proof}

\subsection{Rotation invariance of Dirichlet forms}
Let us introduce the notion of the rotation invariance for regular Dirichlet forms on $L^2({\mathbb R}^d;{\rm d}x)$.
For $Q\in O(d)$ and a real valued function $f$ on ${\mathbb R}^d$, 
we define the function $f\circ Q$ on ${\mathbb R}^d $ by $(f\circ Q)(x)=f(Qx) \ (x\in {\mathbb R}^d)$. 
\begin{defn}
Let $({\cal E},{\cal F})$ be a regular Dirichlet form on $L^2({\mathbb R}^d;{\rm d}x)$. 
Then $({\cal E},{\cal F})$ is rotation invariant if for any $Q\in O(d)$, 
the following two conditions hold{\rm :} 
\begin{itemize}
\item For any $f\in L^2({\mathbb R}^d;{\rm d}x)$, 
$f\circ Q\in{\cal F}$ if and only if $f\in {\cal F}$.
\item For any $f\in {\cal F}$, ${\cal E}(f\circ Q,f\circ Q)={\cal E}(f,f)$.
\end{itemize}
\end{defn} 

We show that the rotation invariance of a regular Dirichlet form yields 
that of the finite-dimensional distributions of an associated symmetric Hunt process.
\begin{lem}\label{lem:rot-inv}
Let $({\cal E},{\cal F})$ be a rotation invariant regular Dirichlet form on $L^2({\mathbb R}^d;{\rm d}x)$, 
and let ${\mathbf M}=(\{X_t\}_{t\ge 0}, \{P_x\}_{x\in {\mathbb R}^d})$ 
be a symmetric Hunt process on ${\mathbb R}^d$ associated with $({\cal E},{\cal F})$. 
Then there exists a properly exceptional Borel set $N\subset {\mathbb R}^d$ 
such that 
\begin{equation}\label{eq:rot-inv}
P_x(X_t\in A)=P_{Qx}(X_t\in QA), \quad x\in {\mathbb R}^d\setminus N, \ t>0, \ A\in {\cal B}({\mathbb R}^d).
\end{equation}
In particular, if ${\mathbf M}$ satisfies the strong Feller property, 
then the equality above is valid for any $x\in {\mathbb R}^d$.
\end{lem}
\begin{proof}
Let $({\cal E},{\cal F})$ be a rotation invariant regular Dirichlet form on $L^2({\mathbb R}^d;{\rm d}x)$, 
and let $\{T_t\}_{t\ge 0}$ be a strongly continuous Markovian semigroup on $L^2({\mathbb R}^d;{\rm d}x)$ 
associated with $({\cal E},{\cal F})$. 
For $Q\in O(d)$, we define 
\[
T_t^Qf(x)=T_t(f\circ Q^{-1})(Qx), \quad f\in L^2({\mathbb R}^d;{\rm d}x), \ x\in {\mathbb R}^d, \ t>0.
\]
Then $\{T_t^{Q}\}_{t\ge 0}$ is also a strongly continuous Markovian semigroup on $L^2({\mathbb R}^d;{\rm d}x)$. 
Moreover, by the rotation invariance of the Lebesgue measure, 
we have for any $f\in L^2({\mathbb R}^d;{\rm d}x)$,
\begin{equation*}
\begin{split}
\int_{{\mathbb R}^d}(f(x)-T_t^Qf(x))f(x)\,{\rm d}x
&=\int_{{\mathbb R}^d}\left\{(f\circ Q^{-1})(Qx)-T_t(f\circ Q^{-1})(Qx)\right\} (f\circ Q^{-1})(Qx)\,{\rm d}x\\
&=\int_{{\mathbb R}^d}\left\{(f\circ Q^{-1})(x)-T_t(f\circ Q^{-1})(x)\right\} (f\circ Q^{-1})(x)\,{\rm d}x.
\end{split}
\end{equation*}
Since $({\cal E},{\cal F})$ is rotation invariant,  we obtain  for any $f\in {\cal F}$, 
\begin{equation*}
\begin{split}
&\lim_{t\rightarrow +0}\frac{1}{t}\int_{{\mathbb R}^d}(f(x)-T_t^Qf(x))f(x)\,{\rm d}x\\
&=\lim_{t\rightarrow +0}
\frac{1}{t}\int_{{\mathbb R}^d}\left\{(f\circ Q^{-1})(x)-T_t(f\circ Q^{-1})(x)\right\} (f\circ Q^{-1})(x)\,{\rm d}x\\
&={\cal E}(f\circ Q^{-1},f\circ Q^{-1})={\cal E}(f,f).
\end{split}
\end{equation*}
Therefore, $T_t^Qf=T_tf \ (t>0)$ for any $f\in L^2({\mathbb R}^d;{\rm d}x)$, 
which implies \eqref{eq:rot-inv}. 

Assume in addition that ${\mathbf M}$ satisfies the strong Feller property. 
Then for any fixed $t>0$ and $A\in {\cal B}({\mathbb R}^d)$, 
the function $x\mapsto P_x(X_t\in A)$ is continuous on ${\mathbb R}^d$. 
We also know that any properly exceptional set is of zero capacity, and so of zero Lebesgue measure. 
Therefore, we have \eqref{eq:rot-inv} for any $x\in {\mathbb R}^d$.
\end{proof}

\begin{lem}\label{lem:rot-heat}
Let $({\cal E},{\cal F})$ and ${\mathbf M}=(\{X_t\}_{t\ge 0}, \{P_x\}_{x\in {\mathbb R}^d})$ be as in Lemma {\rm \ref{lem:rot-inv}}. 
Assume that ${\mathbf M}$ satisfies the strong Feller property.
Let $p_t(x,y)$ be the density of the transition function for ${\mathbf M}$. 
Then for any $x\in {\mathbb R}^d$, $t>0$ and $Q\in O(d)$,
\begin{equation}\label{eq:rot-heat}
p_t(x,y)=p_t(Qx,Qy),  \quad \text{$m$-a.e.\ $y\in {\mathbb R}^d$}.
\end{equation}
If, for any fixed $x\in {\mathbb R}^d$ and $t>0$, the function $y\mapsto p_t(x,y)$ is continuous on ${\mathbb R}^d$, 
then the equality above is valid for any $y\in {\mathbb R}^d$. 
\end{lem}

\begin{proof}
For any $x\in {\mathbb R}^d$, $t>0$ and $A\in {\cal B}({\mathbb R}^d)$, 
we have by the change of variables formula with $z=Qy$,
\[
P_{Qx}(X_t\in QA)=\int_{QA}p_t(Qx,z)\,{\rm d}z=\int_Ap_t(Qx,Qy)\,{\rm d}y.
\]
Since Lemma {\rm \ref{lem:rot-inv}} also yields 
\[
P_{Qx}(X_t\in QA)=P_x(X_t\in A)=\int_A p_t(x,y)\,{\rm d}y,
\]
we have \eqref{eq:rot-heat}. 
This implies the rest of the assertion.
\end{proof}

In general, if ${\mathbf M}_1$ and ${\mathbf M}_2$ are symmetric Hunt processes on ${\mathbb R}^d$ associated 
with a common regular Dirichlet form $({\cal E},{\cal F})$ on $L^2({\mathbb R}^d;{\rm d}x)$, 
then their transition functions coincide outside a common properly exceptional set 
(\cite[Theorem 4.2.8]{FOT11}). 
On the other hand, for a class of regular Dirichlet forms, 
there exist Feller process versions of associated symmetric Hunt processes as in \eqref{eq:realization}. 
Under this setting, if we further assume that the regular Dirichlet form is rotation invariant,  
then Lemmas \ref{lem:rotation} and \ref{lem:rot-heat} imply that for any $Q\in O(d)$, 
$\{P_x\}_{x\in {\mathbb R}^d}$ and $\{P_x^Q\}_{x\in {\mathbb R}^d}$ 
coincide as a family of probability measures on $(\Omega,{\cal F})$.

Let us present examples of rotation invariant regular Dirichlet forms on $L^2({\mathbb R}^d;{\rm d}x)$ 
such that we have a Feller process version of associated symmetric Hunt processes. 

\begin{exam}\label{exam:diffusion}\rm 
Let $({\cal E},{\cal F})$ be a regular Dirichlet form on $L^2({\mathbb R}^d;{\rm d}x)$ as in \eqref{eq:form-diffusion}. 
Let ${\mathbf M}=(\{X_t\}_{t\ge 0},\{P_x\}_{x\in {\mathbb R}^d})$ be an associated symmetric diffusion process on ${\mathbb R}^d$ 
of the form \eqref{eq:realization}, which has the Feller and strong Feller properties. 
Assume that each $a_{ij}(x)$ is rotation invariant, that is, $a_{ij}(Qx)=a_{ij}(x)$ for any $Q\in O(d)$. 
Then $({\cal E},{\cal F})$ satisfies the rotation invariance property. 
In particular, Lemmas \ref{lem:rotation} and \ref{lem:rot-heat} imply that for any $Q\in O(d)$, 
$\{P_x\}_{x\in {\mathbb R}^d}$ and $\{P_x^Q\}_{x\in {\mathbb R}^d}$ 
coincide as a family of probability measures on $(\Omega,{\cal F})$.
\end{exam}

\begin{exam}\label{exam:stable-like}\rm 
Let $({\cal E},{\cal F})$ be a regular Dirichlet form on $L^2({\mathbb R}^d;{\rm d}x)$ as in \eqref{eq:form-stable-like}. 
Let ${\mathbf M}=(\{X_t\}_{t\ge 0},\{P_x\}_{x\in {\mathbb R}^d})$ be an associated symmetric jump process on ${\mathbb R}^d$ 
of the form \eqref{eq:realization}, which has the Feller and strong Feller properties. 
If $c(x,y)$ is rotation invariant, that is, $c(Qx,Qy)=c(x,y)$ for any $Q\in O(d)$, 
then $({\cal E},{\cal F})$ satisfies the rotation invariance property. 
In particular, Lemmas \ref{lem:rotation} and \ref{lem:rot-heat} imply that for any $Q\in O(d)$, 
$\{P_x\}_{x\in {\mathbb R}^d}$ and $\{P_x^Q\}_{x\in {\mathbb R}^d}$ 
coincide as a family of probability measures on $(\Omega,{\cal F})$.
\end{exam}

\subsection{Distributions of positive continuous additive functionals under rotation invariance}

For a Borel measure $\mu$ on ${\mathbb R}^d$ and $Q\in O(d)$, we define the measure $\mu^Q$ on ${\cal B}({\mathbb R}^d)$ by 
\[
\mu^Q(A)=\mu(Q^{-1}A), \quad A\in {\cal B}({\mathbb R}^d).
\]
We then have 
\begin{prop}\label{prop:rot-pcaf}
Assume that $({\cal E},{\cal F})$ is rotation invariant 
and an associated symmetric Hunt process ${\mathbf M}$ is a Feller process. 
Then for any $\mu\in S_{00}$ and $Q\in O(d)$, $\mu^Q\in S_{00}$. 
Moreover, $P_x(A_t^{\mu}\le r)=P_{Qx}(A_t^{\mu^Q}\le r) \ (r\ge 0)$ for any $t>0$ and $x\in {\mathbb R}^d$. 
\end{prop}

\begin{proof}
We realize ${\mathbf M}$ as in \eqref{eq:realization}. 
Let $Q\in O(d)$ and $\mu\in S_{00}$. 
Then $\mu^Q\in S_{00}$ by the rotation invariance of $({\cal E},{\cal F})$ and $G_1(x,y)$. 
If $\Lambda\subset \Omega$ denotes the defining set of $\{A_t^{\mu}\}_{t\ge 0}\in {\mathbf A}_{c,1}^+$, 
then Lemmas \ref{lem:rotation} and \ref{lem:rot-heat} yield 
$P_x^Q(\Lambda)=P_{Qx}(Q\Lambda)=P_x(\Lambda)=1$ for any $x\in {\mathbb R}^d$. 
This also shows that $P_x(Q\Lambda)=1$ for any $x\in {\mathbb R}^d$.  
Therefore, if we define 
\[
A_t^{\mu,Q}(\omega)=A_t^{\mu}(Q^{-1}\omega), \quad \omega\in Q\Lambda, \ t\ge 0,
\]
then $\{A_t^{\mu,Q}\}_{t\ge 0}$ is a positive continuous additive functional in the strict sense 
with the defining set $Q\Lambda$. 
For any $r\ge 0$, since 
\[
Q\{\omega\in \Lambda \mid A_t^{\mu}(\omega)\le r\}=\{\omega\in Q\Lambda \mid A_t^{\mu,Q}(\omega)\le r\},
\]
we have by Lemma \ref{lem:rotation},
\begin{equation}\label{eq:dist}
\begin{split}
P_x(A_t^{\mu}\le r)
=P_{Qx}(Q\{\omega\in \Lambda \mid A_t^{\mu}(\omega)\le r\})
&=P_{Qx}(\{\omega\in Q\Lambda \mid A_t^{\mu,Q}(\omega)\le r\})\\
&=P_{Qx}(A_t^{\mu,Q}\le r).
\end{split}
\end{equation}

On the other hand,  Lemma \ref{lem:rotation} also implies that, 
for any ${\cal F}$-measurable function $F: \Omega\to [0,\infty]$, 
$E_x[F]=E_{Qx}[F(Q^{-1}\cdot)] \ (x\in {\mathbb R}^d)$. 
Hence if we take $F=\int_0^{\infty}e^{-t}\,{\rm d}A_t^{\mu}$, 
then 
\[
E_{Qx}\left[\int_0^{\infty}e^{-t}\,{\rm d}A_t^{\mu,Q}\right]
=E_x\left[\int_0^{\infty}e^{-t}\,{\rm d}A_t^{\mu}\right]
=G_1\mu(x), \quad x\in {\mathbb R}^d.
\]
Since Lemma \ref{lem:rot-heat} yields
\begin{equation*}
\begin{split}
E_{Qx}\left[\int_0^{\infty}e^{-t}\,{\rm d}A_t^{\mu^Q}\right]
&=\int_{{\mathbb R}^d}G_1(Qx,y)\,\mu^Q({\rm d}y)
=\int_{{\mathbb R}^d}G_1(Qx,Qy)\,\mu({\rm d}y)\\
&=\int_{{\mathbb R}^d}G_1(x,y)\,\mu({\rm d}y)=G_1\mu(x), \quad x\in {\mathbb R}^d,
\end{split}
\end{equation*}
we obtain
\[
E_{Qx}\left[\int_0^{\infty}e^{-t}\,{\rm d}A_t^{\mu,Q}\right]
=E_{Qx}\left[\int_0^{\infty}e^{-t}\,{\rm d}A_t^{\mu^Q}\right], \quad x\in {\mathbb R}^d.
\]
Then by \cite[Theorem 5.1.6]{FOT11}, 
there exists a set $\Lambda_0\subset \Omega$ with $P_x(\Lambda_0)=1 \ (x\in {\mathbb R}^d)$ 
such that for any $t\ge 0$ and $\omega\in \Lambda_0$, $A_t^{\mu,Q}(\omega)=A_t^{\mu^Q}(\omega)$. 
Combining this with \eqref{eq:dist}, we arrive at the desired assertion. 
\end{proof}

For $\mu\in S_1$, set $g^{\mu}(x)=E_x[e^{-A_{\infty}^{\mu}}] \ (x\in {\mathbb R}^d)$ as in \eqref{eq:gauge}.   
We then have
\begin{prop}\label{prop:rot-pcaf-1}
Assume that $({\cal E},{\cal F})$ is rotation invariant 
and an associated symmetric Hunt process ${\mathbf M}$ is a Feller process as in \eqref{eq:realization}.
Let $Q\in O(d)$ and $\mu\in S_1$. 
\begin{enumerate}
\item[{\rm (i)}] For any bounded continuous function $f$ on $[0,\infty)$, 
\[
E_{Qx}[f(A_t^{\mu^Q})]=E_x[f(A_t^{\mu})], \quad x\in {\mathbb R}^d, \ t\ge 0.
\]
In particular, if $\mu^Q=\mu$, the the function $x\mapsto E_x[f(A_t^{\mu})]$ is rotation invariant. 
\item[{\rm (ii)}]
For any $x\in {\mathbb R}^d$, $g^{Q\mu}(Qx)=g^{\mu}(x)$
In particular, if $\mu^Q=\mu$, then the function $g^{\mu}(x)$ is rotation invariant. 
\end{enumerate}
\end{prop}
\begin{proof}
Let $Q\in O(d)$ and  $\mu\in S_1$. 
We first prove (i). By \cite[Theorem 5.1.7]{FOT11}, there exists a sequence $\{O_n\}$ of Borel finely open sets 
with $O_n\nearrow E$ such that the measure $\mu_n:={\bf 1}_{O_n}\cdot \mu$ belongs to $S_{00}$ for any $n\ge 1$. 
Hence if $f$ is  a bounded continuous function on $[0,\infty)$, 
then by Proposition \ref{prop:rot-pcaf}, 
\begin{equation}\label{eq:mu-approx}
E_{Qx}[f(A_t^{\mu_n^Q})]=E_x[f(A_t^{\mu_n})], \quad x\in {\mathbb R}^d, \ t\ge 0.
\end{equation}
We here note that as $n\rightarrow\infty$, 
\[
A_t^{\mu_n}(\omega)=\int_0^t {\bf 1}_{O_n}(X_s(\omega))\,{\rm d}A_s^{\mu}(\omega) \nearrow A_t^{\mu}(\omega), \quad \text{$P_x$-a.s.\ $\omega\in \Omega$}
\]
and
\begin{equation*}
\begin{split}
A_t^{\mu_n^Q}(\omega)=A_t^{\mu_n,Q}(\omega)=A_t^{\mu_n}(Q^{-1}\omega)
&=\int_0^t {\bf 1}_{O_n}(X_s(Q^{-1}\omega))\,{\rm d}A_s^{\mu}(Q^{-1}\omega)\\
&\nearrow A_t^{\mu}(Q^{-1}\omega)=A_t^{\mu^Q}(\omega), \quad \text{$P_x$-a.s.\ $\omega\in \Omega$}.
\end{split}
\end{equation*}
Then by the dominated convergence theorem applied to \eqref{eq:mu-approx}, 
we get (i). 

We next show (ii). By (i) with $f(x)=e^{-x}$, we have 
\[
E_{Qx}[e^{-A_t^{\mu^Q}}]=E_x[e^{-A_t^{\mu}}], \quad x\in {\mathbb R}^d, \ t\ge 0.
\]
Letting $t\rightarrow\infty$, we obtain (ii).
\end{proof}

The next theorem is a direct consequence of \cite[Theorem 1]{M77} 
and the rotation invariance property of measures. 
\begin{thm}\label{thm:pot-decay}{\rm (\cite[Theorem 1]{M77})} 
Let $\mu$ be a positive Radon measure on ${\mathbb R}^d$ such that $\mu^Q=\mu$ for any $Q\in O(d)$. 
If for some $x_0\in {\mathbb R}^d$ and $\alpha\in (1,d)$,   
\begin{equation}\label{eq:pot-finite}
\int_{{\mathbb R}^d}|x_0-y|^{\alpha-d}\,\mu({\rm d}y)<\infty,
\end{equation}
then $\lim_{|x|\rightarrow\infty}\int_{{\mathbb R}^d}|x-y|^{\alpha-d}\,\mu({\rm d}y)=0.$
\end{thm}

\begin{rem}\rm 
Hansen-Bogdan \cite[Remark 2.7]{HB25} pointed out 
that \cite[Theorem 2]{M77} is valid only for $\alpha>1$ 
because the surface of a ball needs to  have positive capacity (see also \cite[Example 2.4]{HB25}). 
This comment also applies to \cite[Theorem 1]{M77}. 
\end{rem}

Combining this with Proposition \ref{prop:rot-pcaf-1}, we obtain 
\begin{cor}\label{cor:pot-decay}
\begin{enumerate}
\item[{\rm (i)}] 
Let $({\cal E},{\cal F})$ be as in Example {\rm \ref{exam:diffusion}} with  $d\ge 3$. 
Let $\mu$ be a positive Radon measure on ${\mathbb R}^d$ such that 
$\mu\in S_1$ and $\mu^Q=\mu$ for any $Q\in O(d)$.
Then
\begin{equation}\label{eq:pot-decay-1}
\lim_{|x|\rightarrow\infty}\int_{{\mathbb R}^d}|x-y|^{2-d}g^{\mu}(y)\,\mu({\rm d}y)=0.
\end{equation}
\item[{\rm (ii)}] 
Let $({\cal E},{\cal F})$ be as in Example {\rm \ref{exam:stable-like}} with $\alpha\in (1,d\wedge 2)$.
Let $\mu$ be a positive Radon measure on ${\mathbb R}^d$ 
such that $\mu\in S_1$ and $\mu^Q=\mu$ for any $Q\in O(d)$. 
Then
\begin{equation}\label{eq:pot-decay-2}
\lim_{|x|\rightarrow\infty}\int_{{\mathbb R}^d}|x-y|^{\alpha-d}g^{\mu}(y)\,\mu({\rm d}y)=0.
\end{equation}
\end{enumerate}
\end{cor}

\begin{proof}
Under the setting of (i) or (ii), 
$g^{\mu}$ is rotation invariant by Proposition \ref{prop:rot-pcaf-1} (ii). 
Since $g^{\mu}$ is non-negative and bounded, 
the measure $\nu=g^{\mu}\cdot \mu$ also belongs to $S_1$ and $\nu^Q=\nu$ for any $Q\in O(d)$. 
Moreover, Proposition \ref{prop:g-eq} yields  for any $x\in {\mathbb R}^d$, 
\[
\int_{{\mathbb R}^d}|x-y|^{\alpha-d}\,\nu^Q({\rm d}y)
=\int_{{\mathbb R}^d}|x-y|^{\alpha-d}\,\nu({\rm d}y)
=\int_{{\mathbb R}^d}|x-y|^{\alpha-d}g^{\mu}(y)\,\mu({\rm d}y)<\infty, 
\]
where $\alpha=2$ under the setting of  (i). 
Hence by Theorem \ref{thm:pot-decay}, we have the desired assertion. 
\end{proof}

As an application of Corollary \ref{cor:pot-decay}, 
we verify the condition \eqref{thm:pcaf-div} for  Examples \ref{exam:absolute} and \ref{exam:singular}. 

\begin{rem}\rm 
Let $f$ be a positive Borel measurable function on ${\mathbb R}^d$ such that 
for some positive constants  $c_1$ and $c_2$, $c_1\le f(x)\le c_2 \ (x\in {\mathbb R}^d)$. 
Let $\mu$ be a positive Radon measure on $E$ such that $\mu\in S_1$ and $\mu^Q=\mu$ for any $Q\in O(d)$.  
If we define the measure $\mu_f$ by $\mu_f=f\cdot\mu$, 
then $\mu_f$ is also a positive Radon measure on ${\mathbb R}^d$ and belongs to $S_1$. 
Since 
\[
g^{c_2\mu}(x)\le g^{\mu_f}(x) \le g^{c_1\mu}(x), \quad x\in {\mathbb R}^d,
\]
\eqref{eq:pot-decay-1} and \eqref{eq:pot-decay-2} are valid with $\mu$ replaced by $\mu_f$. 
This shows that, even if $\mu_f$ is rotation variant, 
we can verify the condition \eqref{thm:pcaf-div} for $\mu=\mu_f$. 
In particular,  under the setting of Example \ref{exam:absolute} or Example \ref{exam:singular}, 
we can apply Theorem \ref{thm:liouville} to the Schr\"odinger operators with the potential term $\mu_f$.
\end{rem}

\end{document}